\documentclass[11pt]{amsart}
\usepackage[T1]{fontenc}
\usepackage[utf8]{inputenc}
\usepackage{lmodern}
\usepackage{tikz}
\usetikzlibrary{positioning,arrows.meta}
\usepackage{graphicx}
\newtheorem*{introcorollary}{Corollary}
\usepackage{amsmath,amssymb,amsthm,mathtools}
\usepackage{microtype}
\usepackage[hidelinks]{hyperref}
\hypersetup{
  pdftitle={Determinant characters and multiplicative solubility of finite skew braces},
  pdfauthor={Marco Damele},
  pdfsubject={Solubility of finite skew braces with nilpotent additive derived subgroup},
  pdfkeywords={skew braces, soluble groups, nilpotent derived subgroup, Frattini argument}
}

\newtheorem{maintheorem}{Theorem}

\newtheorem{theorem}{Theorem}[section]
\newtheorem{lemma}[theorem]{Lemma}
\newtheorem{proposition}[theorem]{Proposition}
\newtheorem{corollary}[theorem]{Corollary}
\theoremstyle{definition}
\newtheorem{definition}[theorem]{Definition}
\theoremstyle{remark}

\DeclareMathOperator{\Aut}{Aut}

\DeclareMathOperator{\GL}{GL}

\DeclareMathOperator{\Syl}{Syl}

\DeclareMathOperator{\Ad}{Ad}
\newcommand{\F}{\mathbf F}
\newcommand{\N}{\mathrm N}

\numberwithin{equation}{section}
\allowdisplaybreaks[1]
\title[Solubility of finite skew braces]{Determinant characters and multiplicative solubility of finite skew braces}
\author{Marco Damele}
\address{Department of Mathematics and Computer Science, University of Cagliari, Italy}
\email{marco.damele@unica.it}
\subjclass[2020]{Primary 16T25; Secondary 20D10, 20D15, 20B35}
\keywords{Finite skew brace, soluble group, nilpotent derived subgroup, supersoluble group, Frattini argument, holomorph}
\date{}

\begin{document}

\begin{abstract}

Let $B$ be a finite skew brace. We prove that if the
additive derived subgroup $(B,+)'$ is nilpotent, then
the multiplicative group $(B,\cdot)$ is soluble.
In particular, this holds for finite skew braces
with supersoluble or metabelian additive group.
\end{abstract}

\maketitle

\section{Introduction}\label{sec:introduction}

A \emph{skew brace} is a triple $(B,+,\cdot)$ such that
$(B,+)$ and $(B,\cdot)$ are groups and for every $a,b,c \in B$:
\[
 a\cdot(b+c)=(a\cdot b)-a+(a\cdot c)
\]
The two groups are called the \emph{additive} and
\emph{multiplicative} groups of $B$, respectively.
Introduced by Guarnieri and Vendramin~\cite{GV} as a
generalisation of Rump's braces~\cite{Rump}, skew braces
arose in the study of set-theoretic solutions of the
Yang--Baxter equation. They also connect this subject
with regular subgroups of holomorphs and Hopf--Galois
theory~\cite[Appendix~A]{SV}.
Throughout this paper, all groups and skew braces are finite. A central problem is to understand how the structure
of either group constrains the other.
This has been studied through solubility
criteria~\cite{Byott2015,GN,Nasybullov,TQ}, restrictions
involving simple groups~\cite{Byott2024,TsangSimple,DameleST},
and comparisons of Sylow structure and derived
subgroups~\cite{DameleZ,DameleSimpleZ,LNZ}.
When the additive group is abelian, the multiplicative
group is soluble, as follows from the work of Etingof,
Schedler and Soloviev~\cite{ESS}.
Byott~\cite[Theorem~1]{Byott2015} proved the corresponding
result for nilpotent additive groups in the language
of Hopf--Galois theory; it was subsequently obtained
for skew braces by Smoktunowicz and
Vendramin~\cite[Corollary~2.23]{SV}.
The natural extension to soluble additive groups was
proposed by Byott~\cite{Byott2015} and formulated for
skew braces in~\cite[Question~2.25]{SV} and~\cite[Problem~2.46]{Vendramin}.
Known as \emph{Byott's conjecture}, it remained open
for more than a decade.
Di Matteo, Ferrara and Trombetti~\cite{DFT} recently
disproved it by constructing a skew brace with soluble
additive group and insoluble multiplicative group.
Our main result shows that nilpotency of the additive
derived subgroup is sufficient.

\begin{maintheorem}\label{thm:A}
Let $B$ be a finite skew brace.
If $(B,+)'$ is nilpotent, then $(B,\cdot)$ is soluble.
\end{maintheorem}

In particular, this extends the result for cyclic
additive derived subgroups in~\cite{LNZ}.
The proof uses Byott's description of a minimal
counterexample~\cite[Theorem~3.2]{Byott2024}.
Its multiplicative group has a soluble normal subgroup
$R$ with
\[
 (B,\cdot)/R\simeq\GL_3(2)^r\rtimes T,
\]
where $r\geq1$ and $T\leq S_r$ permutes the factors.
We attach to this quotient a character
$ \overline\psi:(B,\cdot)/R\longrightarrow\ \F_7^\times$.
For a suitable element $z\in\GL_3(2)^r$, perfectness
forces $\overline\psi(z)=1$, whereas a determinant
calculation gives $\overline\psi(z)=4$.
This contradiction excludes the minimal counterexample. Theorem~A applies in particular to supersoluble groups,
a class lying strictly between nilpotent and soluble
groups. Their derived subgroups are
nilpotent~\cite[Theorem~5.4.10]{Robinson}.
It also applies to metabelian groups, whose derived
subgroups are abelian.

\begin{introcorollary}
Let $B$ be a finite skew brace with supersoluble additive
group. Then its multiplicative group is soluble.
\end{introcorollary}

\begin{introcorollary}
Let $B$ be a finite skew brace with metabelian additive
group. Then its multiplicative group is soluble.
\end{introcorollary}

Theorem \ref{thm:A} applies to additive groups that are neither
supersoluble nor metabelian. For instance, consider $G=\mathrm{SL}_2(3)\cong Q_8\rtimes C_3.$
Its derived subgroup is isomorphic to the quaternion
group $Q_8$, which is nilpotent but non-abelian.
Moreover, $G$ is not supersoluble, since
$G/Z(G)\cong A_4$ is not supersoluble.
Thus $G$ satisfies the hypothesis of Theorem \ref{thm:A}
without satisfying either of the two conditions
in the preceding corollaries.

The correspondence between skew braces and
regular subgroups of holomorphs translates Theorem~A
into a restriction on Hopf--Galois types: a Galois
extension with group $G$ admits a structure of type
$N$ if and only if there is a skew brace with additive
group $N$ and multiplicative group
$G$~\cite{GP,Byott1996,SV}.

\begin{introcorollary}
Let $L/K$ be a finite Galois extension admitting a
Hopf--Galois structure of type $N$.
If $N'$ is nilpotent, then $\operatorname{Gal}(L/K)$
is soluble. In particular, this holds when $N$ is
supersoluble or metabelian.
\end{introcorollary}

The paper is organised as follows.
Section~\ref{sec:prelim} collects the preliminaries,
reduces Theorem~\ref{thm:A} to the main obstruction,
and establishes its principal consequences.
Section~\ref{sec:characters} introduces determinant
characters, while Section~\ref{sec:normal-sylows}
uses normal additive Sylow subgroups to construct
characters on multiplicative quotients.
Finally, Section~\ref{sec:obstruction} proves the
obstruction and completes the proof of Theorem~\ref{thm:A}.

\section{Preliminaries and reduction}\label{sec:prelim}

\subsection{Skew braces}
A \emph{skew brace} is a triple $(B,+,\cdot)$ such that both
$(B,+)$ and $(B,\cdot)$ are groups and for all $a,b,c \in B$ the following compatibility condition holds:
\[
 a\cdot(b+c)=(a\cdot b)-a+(a\cdot c)
\]
The map $\lambda:(B,\cdot)\to\Aut(B,+)$
sending $a \in B$ to $\lambda_a$ where
\[
 \lambda_a(b)=-a+(a\cdot b)
\]
is a well defined group homomorphism (see \cite{GV}).
A \emph{subbrace} of $(B,+,\cdot)$ is a subset $C \subset B$ such that $(C,+) \le (B,+)$ and $(C,\cdot) \le (B,\cdot)$. If $(C,+)$ is a characteristic subgroup of $(B,+)$, then $C$ is a subbrace of $B$ since for every $x,y \in C$ $x \cdot y=x+\lambda_x(y) \in C$. A subbrace $L$ of $B$ is called a \emph{left ideal}
if $ \lambda_b(L)\subseteq L$ for every $b \in B$;
see \cite{CSV}. A left ideal $I$ of $B$ is called an \emph{ideal}, and we write $I \trianglelefteq B$, if $ (I,+)\trianglelefteq(B,+),  (I,\cdot)\trianglelefteq(B,\cdot)$;
see \cite[Definition~5.1 and Lemma~5.5]{SV}. If $I$ is an ideal, the additive and multiplicative
cosets coincide, since
\[
 b\cdot I=b+\lambda_b(I)=b+I
 \qquad(b\in B).
\]
Thus the quotient set $ B/I=\{b+I:b\in B\}$
inherits the operations
\[
 \begin{aligned}
 (a+I)+(b+I)&=(a+b)+I,\\
 (a+I)\cdot(b+I)&=(a\cdot b)+I.
 \end{aligned}
\]
These operations are well defined and make $B/I$
a skew brace, called the \emph{quotient skew brace}
of $B$ by $I$.
A map $\alpha:B \longrightarrow B$ is called an automorphism of the skew brace $B$ if $\alpha \in \operatorname{Aut}(B,+) \cap \operatorname{Aut}(B,\cdot)$.
For a group $G$, write $G'$ for its derived subgroup, $\Phi(G)$ for
its Frattini subgroup, and $\N_G(H)$ for the normalizer of a subgroup
$H$. For $x,y \in G$ we put $\Ad_x(y)=xyx^{-1}.$

\subsection{Byott's reduction}

As recalled in the introduction, Byott's conjecture
asks whether every finite skew brace with soluble
additive group has soluble multiplicative group.
Although the conjecture is false~\cite{DFT},
Byott's description of a minimal counterexample
provides the starting point of our proof.
We state his reduction in the language of skew braces.

\begin{theorem}[{\cite[Theorem~3.2]{Byott2024}}]
\label{thm:Byott}
Let $B$ be a finite skew brace such that
$(B,+)$ is soluble and $(B,\cdot)$ is insoluble.
Suppose that every proper subbrace of $B$ has
soluble multiplicative group.
Then every maximal normal subgroup of $(B,+)$
has index $2$, and there exist a soluble normal
subgroup $R$ of $(B,\cdot)$, an integer $r\geq1$,
and a soluble transitive subgroup $T\leq S_r$
such that
\[
 (B,\cdot)/R\simeq\GL_3(2)^r\rtimes T,
\]
where $T$ acts by permuting the $r$ factors.
\end{theorem}

\begin{corollary}\label{cor:Byott-abelianization}
Under the hypotheses of Theorem~\ref{thm:Byott},
the group $(B,+)/(B,+)'$ is a $2$-group.
\end{corollary}

\begin{proof}
Let $p$ be a prime divisor of $|(B,+)/(B,+)'|$.
Since this quotient is abelian, it has a normal
subgroup $N/(B,+)'$ of index $p$, where
$N\trianglelefteq(B,+)$.
Then $[(B,+):N]=p$, so $N$ is a maximal normal
subgroup of $(B,+)$.
Thus $p=2$ by Theorem~\ref{thm:Byott}.
\end{proof}
\subsection{Proof of Theorem~\ref{thm:A}}
We first reduce Theorem~\ref{thm:A} to the following key obstruction, which will be proved in Section~\ref{sec:obstruction}.
\begin{proposition}\label{prop:obstruction}
Let $B$ be a finite skew brace whose additive Sylow $3$- and $7$-subgroups are normal. Then $(B,\cdot)$ has no soluble normal subgroup $R$ such that
$(B,\cdot)/R\simeq\GL_3(2)^r\rtimes T,$
where $r\geq1$ and $T\leq S_r$ permutes the factors.
\end{proposition}
\begin{proof}[Proof of Theorem~\ref{thm:A}]
Suppose otherwise, and choose a counterexample $B$ of minimal order. Since $(B,+)'$ is nilpotent, $(B,+)$ is soluble. For every proper subbrace $C$ of $B$, the group $(C,+)'\leq(B,+)'$ is nilpotent, so $(C,\cdot)$ is soluble by minimality. By Theorem~\ref{thm:Byott}, there exists a soluble normal subgroup $R$ of $(B,\cdot)$ with $(B,\cdot)/R\simeq\GL_3(2)^r\rtimes T$, and $(B,+)/(B,+)'$ is a $2$-group by Corollary \ref{cor:Byott-abelianization}. Thus, for $p\in\{3,7\}$, every Sylow $p$-subgroup
of $(B,+)$ lies in $(B,+)'$ and is a Sylow
$p$-subgroup of $(B,+)'$.
Since $(B,+)'$ is nilpotent, this Sylow subgroup
is unique and hence characteristic in $(B,+)'$.
Moreover, $(B,+)'$ is characteristic in $(B,+)$,
so the Sylow subgroup is characteristic, and
therefore normal, in $(B,+)$.
Hence both additive Sylow subgroups are normal
in $(B,+)$, contradicting
Proposition~\ref{prop:obstruction}.
\end{proof}
The proof of Theorem~\ref{thm:A} thus reduces to Proposition~\ref{prop:obstruction}. We construct a character $\overline\psi_{R,7,3}:(B,\cdot)/R\to\F_7^\times$ and evaluate it on a suitable element $z\in\GL_3(2)^r$. Since this subgroup is perfect, $\overline\psi_{R,7,3}(z)=1$, whereas the determinant calculation gives $\overline\psi_{R,7,3}(z)=4$. We can summarise the logical steps of the proof as follows:
\medskip 

\begin{center}
\setlength{\fboxsep}{7pt}
\fbox{$\displaystyle (B,\cdot)/R\simeq\GL_3(2)^r\rtimes T\ \rightsquigarrow\ \overline\psi_{R,7,3}\ \rightsquigarrow\ 1=4\text{ in }\F_7^\times$}
\end{center}
\subsection{Consequences of Theorem~\ref{thm:A}}

Before proving Proposition~\ref{prop:obstruction},
we record some immediate consequences of
Theorem~\ref{thm:A}.

\begin{corollary}\label{cor:intro-super}
Let $B$ be a finite skew brace with supersoluble additive group. Then $(B,\cdot)$ is soluble.
\end{corollary}
\begin{proof}
The derived subgroup of a finite supersoluble group is nilpotent~\cite[Theorem~5.4.10]{Robinson}. Apply Theorem~\ref{thm:A}.
\end{proof}
\begin{corollary}\label{cor:metabelian}
Let $B$ be a finite skew brace with metabelian additive group. Then $(B,\cdot)$ is soluble.
\end{corollary}
\begin{proof}
The subgroup $(B,+)'$ is abelian, hence nilpotent. Apply Theorem~\ref{thm:A}.
\end{proof}
By the Greither--Pareigis correspondence~\cite{GP} and Byott's translation~\cite{Byott1996}, a finite Galois extension $L/K$ admits a Hopf--Galois structure of type $N$ if and only if there is a skew brace with additive group isomorphic to $N$ and multiplicative group isomorphic to $\operatorname{Gal}(L/K)$; see~\cite[Appendix~A]{SV}.
\begin{corollary}\label{cor:Hopf}
Let $L/K$ be a finite Galois extension admitting a Hopf--Galois structure of type $N$. If $N'$ is nilpotent, then $\operatorname{Gal}(L/K)$ is soluble.
\end{corollary}
\begin{proof}
The correspondence above gives a finite skew brace $B$ with $(B,+)\simeq N$ and $(B,\cdot)\simeq\operatorname{Gal}(L/K)$. Apply Theorem~\ref{thm:A}.
\end{proof}
We now develop the determinant characters needed to prove Proposition~\ref{prop:obstruction}.
\section{Determinant characters} \label{sec:characters}
\subsection{Determinants on finite \texorpdfstring{$p$}{p}-groups}
Let $P$ be a finite non-trivial $p$-group. By \cite[Theorem 5.3.2]{Robinson} $P/\Phi(P)$ is an elementary
abelian $p$-group and can be regarded as a vector space over
$\F_p$. Define the iterated Frattini subgroups by
\[
 P_0=\Phi^0(P)=P,\qquad
 P_{i+1}=\Phi^{i+1}(P)=\Phi(P_i)
 \quad(i\geq0).
\]
Since the Frattini subgroup of a group is characteristic,
each $P_i$ is characteristic in $P$. Moreover, the Frattini
subgroup of a nontrivial finite $p$-group is proper. Hence
there exists an integer $m\geq0$ such that $P_m=1$, and we
obtain a characteristic series
\[
 P=P_0\geq P_1\geq\cdots\geq P_m=1.
\]
For every $0\leq i<m$, the quotient $ V_i=P_i/P_{i+1}=P_i/\Phi(P_i)$
is an elementary abelian $p$-group, viewed as an
$\F_p$-vector space. 
Let $\alpha\in\Aut(P)$. Since $\alpha$ preserves both
$P_i$ and $P_{i+1}$, it induces an automorphism
\[
 \alpha_i:V_i\longrightarrow V_i,\qquad
 \alpha_i(xP_{i+1})=\alpha(x)P_{i+1}.
\]

We define
\begin{equation}\label{eq:Delta-definition}
 \Delta_P:\Aut(P)\longrightarrow\F_p^\times,\qquad
 \Delta_P(\alpha)=
 \prod_{i=0}^{m-1}\det_{\F_p}(\alpha_i).
\end{equation}
If $P=1$, this is the empty product, with value $1$.
Equivalently, we may take the product over all $i\geq0$,
assigning determinant $1$ to every trivial factor. Finally, for $\alpha,\beta\in\Aut(P)$, the automorphism
induced by $\alpha\beta$ on $V_i$ is $\alpha_i\beta_i$.
The multiplicativity of the determinant therefore gives
\[
 \begin{aligned}
 \Delta_P(\alpha\beta)
 &=\prod_{i=0}^{m-1}\det_{\F_p}(\alpha_i\beta_i)=\left(\prod_{i=0}^{m-1}\det_{\F_p}(\alpha_i)\right)
   \left(\prod_{i=0}^{m-1}\det_{\F_p}(\beta_i)\right)=\Delta_P(\alpha)\Delta_P(\beta).
 \end{aligned}
\]
Hence $\Delta_P$ is a group homomorphism. 

We first establish some properties of the $\Delta$-map
that will be useful in the subsequent proofs.

\begin{lemma}\label{lem:Delta-series}
Let $P$ be a finite $p$-group and $\alpha\in\Aut(P)$. Then:
\begin{enumerate}
\setlength{\itemsep}{0pt}
\setlength{\parsep}{0pt}
\item For every isomorphism $f:P\to Q$, we have
$\Delta_Q(f\alpha f^{-1})=\Delta_P(\alpha)$.
\item If $I\trianglelefteq P$ is $\alpha$-invariant, then
\begin{equation}\label{eq:Delta-normal}
\Delta_P(\alpha)=\Delta_I(\alpha|_I)\Delta_{P/I}(\overline\alpha),
\end{equation}
where $\overline\alpha$ is induced on $P/I$.
\end{enumerate}
\end{lemma}
\begin{proof}
\setlength{\abovedisplayskip}{2pt}
\setlength{\belowdisplayskip}{2pt}
\setlength{\abovedisplayshortskip}{2pt}
\setlength{\belowdisplayshortskip}{2pt}
\setlength{\jot}{1pt}
For subgroups $Y\trianglelefteq X\leq P$ satisfying
$\alpha(X)=X$ and $\alpha(Y)=Y$, denote by
$\alpha_{X/Y}$ the induced automorphism
\[
\alpha_{X/Y}:X/Y\longrightarrow X/Y,
\qquad xY\longmapsto\alpha(x)Y.
\]
For \textup{(1)}, since $f$ maps maximal subgroups of $P$ onto maximal subgroups of $Q$, we have $f(\Phi^i(P))=\Phi^i(Q)$ for every $i\geq0$. Thus $f$ induces the isomorphisms
\[
\begin{aligned}
f_i:\Phi^i(P)/\Phi^{i+1}(P)
&\longrightarrow\Phi^i(Q)/\Phi^{i+1}(Q),\\
x\Phi^{i+1}(P)&\longmapsto f(x)\Phi^{i+1}(Q).
\end{aligned}
\]
Here $\alpha_i$ is an automorphism of $\Phi^i(P)/\Phi^{i+1}(P)$, whereas $f_i\alpha_i f_i^{-1}$ is an automorphism of $\Phi^i(Q)/\Phi^{i+1}(Q)$. For every $y\in\Phi^i(Q)$,
\[
\begin{aligned}
(f_i\alpha_i f_i^{-1})(y\Phi^{i+1}(Q))
&=f_i\bigl(\alpha(f^{-1}(y))\Phi^{i+1}(P)\bigr)\\
&=f(\alpha(f^{-1}(y)))\Phi^{i+1}(Q)\\
&=(f\alpha f^{-1})_i(y\Phi^{i+1}(Q)).
\end{aligned}
\]
Hence $(f\alpha f^{-1})_i=f_i\alpha_i f_i^{-1}$, and therefore
\[
\Delta_Q(f\alpha f^{-1})
=\prod_i\det(f_i\alpha_i f_i^{-1})
=\prod_i\det(\alpha_i)
=\Delta_P(\alpha).
\]
For \textup{(2)}, we argue by induction on $|P|$. The cases $I=1$ and $I=P$ are immediate. Suppose $1<I<P$ and put $F=\Phi(P)$. By \cite[Theorem~5.3.2]{Robinson},
\[
\Phi(P/I)=(P/I)'(P/I)^p=P'P^pI/I=FI/I.
\]
Consider the isomorphisms
\[
\begin{aligned}
\theta:I/(I\cap F)&\longrightarrow FI/F,
&x(I\cap F)&\longmapsto xF,\\
\eta:F/(I\cap F)&\longrightarrow FI/I,
&y(I\cap F)&\longmapsto yI.
\end{aligned}
\]
Their definitions give
\[
\begin{aligned}
\theta\,\alpha_{I/(I\cap F)}\,\theta^{-1}
&=\alpha_{FI/F},\\
\eta\,\alpha_{F/(I\cap F)}\,\eta^{-1}
&=\alpha_{FI/I}.
\end{aligned}
\]
Since $|I|,|F|<|P|$, induction followed by \textup{(1)} yields
\[
\begin{aligned}
\Delta_I(\alpha)
&=\Delta_{I\cap F}(\alpha)\Delta_{I/(I\cap F)}(\alpha)\\
&\overset{\textup{(1)}}{=}
  \Delta_{I\cap F}(\alpha)\Delta_{FI/F}(\alpha)\\
&=\Delta_{I\cap F}(\alpha)\det(\alpha_{FI/F}),\\
\Delta_F(\alpha)
&=\Delta_{I\cap F}(\alpha)\Delta_{F/(I\cap F)}(\alpha)\\
&\overset{\textup{(1)}}{=}
  \Delta_{I\cap F}(\alpha)\Delta_{FI/I}(\alpha).
\end{aligned}
\]
Here $FI/F$ is elementary abelian. Since $P/F$ is an elementary abelian $p$-group, we regard it as an $\F_p$-vector space. The subgroup $FI/F$ is an $\alpha_{P/F}$-invariant subspace. Choose a basis $v_1,\ldots,v_k$ of $FI/F$ and extend it to a basis
\[
\mathcal B=(v_1,\ldots,v_k,w_1,\ldots,w_\ell)
\]
of $P/F$. Since $\alpha_{P/F}(v_j)\in FI/F$, the images of the first $k$ basis vectors have no components along the $w_i$. Hence
\[
[\alpha_{P/F}]_{\mathcal B}
=\begin{pmatrix}A&C\\0&D\end{pmatrix},
\]
where $A$ represents $\alpha_{FI/F}$.
The images of $w_1,\ldots,w_\ell$ form a basis of
\[
(P/F)/(FI/F)\simeq P/FI.
\]
On this quotient, the components along the $v_i$ vanish, so $D$ represents $\alpha_{P/FI}$. Therefore
\[
\det A=\det(\alpha_{FI/F}),
\qquad
\det D=\det(\alpha_{P/FI}).
\]
Consequently,
\[
\det(\alpha_{P/F})
=\det\begin{pmatrix}A&C\\0&D\end{pmatrix}
=\det A\,\det D
=\det(\alpha_{FI/F})\det(\alpha_{P/FI}).
\]
Consequently,
\[
\begin{aligned}
\Delta_P(\alpha)
&=\det(\alpha_{P/F})\Delta_F(\alpha)\\
&=\det(\alpha_{FI/F})\det(\alpha_{P/FI})
  \Delta_{I\cap F}(\alpha)\Delta_{FI/I}(\alpha)\\
&=\Delta_I(\alpha)\det(\alpha_{P/FI})\Delta_{FI/I}(\alpha)\\
&=\Delta_I(\alpha)\Delta_{P/I}(\overline\alpha).
\end{aligned}
\]
To justify the last equality, recall that the iterated Frattini series of $P/I$ begins with
\[
P/I\geq\underbrace{FI/I}_{\Phi(P/I)}
\geq\Phi(FI/I)\geq\cdots.
\]
Let $\overline\alpha_0$ denote the automorphism induced on $(P/I)/(FI/I)$. By definition, separating the first determinant from the remaining factors gives
\[
\Delta_{P/I}(\overline\alpha)
=\det(\overline\alpha_0)\Delta_{FI/I}(\alpha).
\]
Consider the isomorphism
\[
\rho:(P/I)/(FI/I)\longrightarrow P/FI,
\qquad (xI)(FI/I)\longmapsto xFI.
\]
For every $x\in P$, we have
\[
\begin{aligned}
\rho\!\left(\overline\alpha_0((xI)(FI/I))\right)
&=\rho((\alpha(x)I)(FI/I))\\
&=\alpha(x)FI\\
&=\alpha_{P/FI}\!\left(\rho((xI)(FI/I))\right).
\end{aligned}
\]
Hence
\[
\alpha_{P/FI}=\rho\,\overline\alpha_0\,\rho^{-1},
\qquad
\det(\alpha_{P/FI})=\det(\overline\alpha_0).
\]
Substituting into the preceding expression yields
\[
\Delta_{P/I}(\overline\alpha)
=\det(\alpha_{P/FI})\Delta_{FI/I}(\alpha).
\]
\end{proof}

\begin{proposition}\label{prop:p-brace-determinant}
Let $C$ be a skew brace of $p$-power order and let $\alpha$ be an automorphism of $C$. Then
\begin{equation}\label{eq:p-brace-determinant}
\Delta_{(C,+)}(\alpha)=\Delta_{(C,\cdot)}(\alpha).
\end{equation}
\end{proposition}
\begin{proof}
We argue by induction on $|C|$. Write $C*C$ for the additive subgroup generated by $x*y=-x+x\cdot y-y$, where $x,y\in C$. By~\cite[Proposition~2.3]{CSV}, $C*C$ is an ideal of $C$ and the quotient $C/(C*C)$ is a trivial skew brace. Moreover, a skew brace of prime-power order is left nilpotent~\cite[Proposition~4.4]{CSV}; thus $C*C\neq C$ if $C\neq1$. The automorphism $\alpha$ preserves $C*C$. Applying Lemma~\ref{lem:Delta-series}\textup{(2)} to the two group structures yields
\[
\begin{aligned}
\Delta_{(C,+)}(\alpha)&=\Delta_{(C*C,+)}(\alpha)\Delta_{(C/(C*C),+)}(\overline\alpha),\\
\Delta_{(C,\cdot)}(\alpha)&=\Delta_{(C*C,\cdot)}(\alpha)\Delta_{(C/(C*C),\cdot)}(\overline\alpha).
\end{aligned}
\]
The first factors agree by induction and the second agree because the quotient is trivial.
\end{proof}
\subsection{A character associated with a soluble normal subgroup}
\begin{lemma}\label{lem:soluble-normal-character}
Let $R$ be a soluble normal subgroup of a finite group $G$, and let $p$ be a prime. There exists a homomorphism $\beta_{R,p}:G\to\F_p^\times$ such that, for every $P\in\Syl_p(G)$ and $h\in\N_G(P)$,
\begin{equation}\label{eq:radical-character}
\beta_{R,p}(h)=\Delta_{P\cap R}(\Ad_h|_{P\cap R}).
\end{equation}
Moreover,
\begin{equation}\label{eq:quotient-determinant}
\Delta_P(\Ad_h|_P)
=\beta_{R,p}(h)\Delta_{PR/R}(\Ad_{hR}|_{PR/R}).
\end{equation}
\end{lemma}
\begin{proof}
Refining the derived series of $R$ by the primary and power filtrations of its abelian factors, choose a series $1=R_s\leq\cdots\leq R_0=R,$
where $R_i\trianglelefteq G$ and $R_i/R_{i+1}\simeq\F_{q_i}^{\,d_i},$
with each $q_i$ a prime. For $g\in G$, define
\[
\rho_i(g):R_i/R_{i+1}\longrightarrow R_i/R_{i+1},
\qquad xR_{i+1}\longmapsto gxg^{-1}R_{i+1},
\]
and put
\[
\beta_{R,p}(g)
=\prod_{\substack{0\leq i<s\\q_i=p}}
\det_{\F_p}(\rho_i(g)).
\]
Since $\rho_i(gg')=\rho_i(g)\rho_i(g')$, we have
$\beta_{R,p}(gg')=\beta_{R,p}(g)\beta_{R,p}(g')$.
Fix $P\in\Syl_p(G)$ and $h\in\N_G(P)$. Since $R_i\trianglelefteq G$, we have $P\cap R_i\in\Syl_p(R_i)$. Therefore
\[
\left|\frac{P\cap R_i}{P\cap R_{i+1}}\right|
=\frac{|R_i|_p}{|R_{i+1}|_p}
=\begin{cases}
p^{d_i},&q_i=p,\\
1,&q_i\neq p.
\end{cases}
\]
Let $\sigma_i$ be the automorphism induced by $\Ad_h$ on $(P\cap R_i)/(P\cap R_{i+1})$, namely
\[
\sigma_i\bigl(x(P\cap R_{i+1})\bigr)
=hxh^{-1}(P\cap R_{i+1}).
\]
If $q_i\neq p$, this factor is trivial and $\det(\sigma_i)=1$.
If $q_i=p$, the map
\[
\theta_i:\frac{P\cap R_i}{P\cap R_{i+1}}
\longrightarrow R_i/R_{i+1},
\qquad x(P\cap R_{i+1})\longmapsto xR_{i+1},
\]
is injective, and both groups have order $p^{d_i}$; hence it is an isomorphism. Moreover,
\[
\begin{aligned}
\theta_i\sigma_i\bigl(x(P\cap R_{i+1})\bigr)
&=hxh^{-1}R_{i+1}\\
&=\rho_i(h)\theta_i\bigl(x(P\cap R_{i+1})\bigr).
\end{aligned}
\]
Thus $\rho_i(h)=\theta_i\sigma_i\theta_i^{-1}$ and
$\det(\sigma_i)=\det(\rho_i(h))$.
Iterating Lemma~\ref{lem:Delta-series}\textup{(2)} along $1=P\cap R_s\leq\cdots\leq P\cap R_0=P\cap R$
now gives
\[
\begin{aligned}
\Delta_{P\cap R}(\Ad_h|_{P\cap R})
&=\prod_{i=0}^{s-1}\det_{\F_p}(\sigma_i)=\prod_{\substack{0\leq i<s\\q_i=p}}
  \det_{\F_p}(\rho_i(h))
=\beta_{R,p}(h),
\end{aligned}
\]
proving \eqref{eq:radical-character}. Finally, the isomorphism
\[
\theta:P/(P\cap R)\longrightarrow PR/R,
\qquad x(P\cap R)\longmapsto xR,
\]
satisfies $\theta\,\overline{\Ad_h}\,\theta^{-1}
=\Ad_{hR}|_{PR/R},$
where $\overline{\Ad_h}$ is induced on $P/(P\cap R)$.
Applying Lemma~\ref{lem:Delta-series}\textup{(2)}, then
\textup{(1)} and \eqref{eq:radical-character}, yields
\[
\begin{aligned}
\Delta_P(\Ad_h|_P)
&=\Delta_{P\cap R}(\Ad_h|_{P\cap R})
  \Delta_{P/(P\cap R)}(\overline{\Ad_h})\\
&=\beta_{R,p}(h)\Delta_{PR/R}(\Ad_{hR}|_{PR/R}).
\end{aligned}
\]
\end{proof}
\section{Normal additive Sylow subgroups and quotient characters}\label{sec:normal-sylows}
\subsection{Normalizers of Sylow subbraces}\label{subsec:sylow-normalizers}
We study multiplicative normalizers when two normal additive Sylow subgroups have coprime orders and hence commute.
\begin{lemma}\label{lem:brace-normalizer}
Let $B$ be a finite skew brace and let $p,q$ be
distinct primes dividing $|B|$.
Suppose that $(P,+)$ and $(Q,+)$ are normal Sylow
$p$- and $q$-subgroups of $(B,+)$, respectively.
Then:
\begin{enumerate}
\item
$\N_{(B,\cdot)}\bigl((P,\cdot)\bigr)$
is a subbrace of $B$.

\item
$\N_{(B,\cdot)}\bigl((P,\cdot)\bigr)\cap(Q,\cdot)
 \in
 \Syl_q\!\left(
 \N_{(B,\cdot)}\bigl((P,\cdot)\bigr)
 \right).$

\item
For every $ t\in
 \N_{(B,\cdot)}\bigl((P,\cdot)\bigr)\cap(Q,\cdot),$
the restriction $\lambda_t|_P$ is an automorphism
of the skew brace $P$, and
\begin{equation}\label{eq:normalizer-action}
 \lambda_t(x)=t\cdot x\cdot t^{-1}
 \qquad(x\in P),
\end{equation}
where $t^{-1}$ denotes the inverse in $(B,\cdot)$.
\end{enumerate}
\end{lemma}

\begin{proof}
\textup{(1)}
Since $(P,+)$ and $(Q,+)$ are normal Sylow subgroups
of $(B,+)$, they are characteristic in $(B,+)$.
Thus $P$ and $Q$ are left ideals and subbraces of $B$.
By \cite[Theorem~A(2)]{CampMoraPerezAlemanyPerezAltarriba},
the intersection $ \N_{(B,+)}\bigl((P,+)\bigr)
 \cap
 \N_{(B,\cdot)}\bigl((P,\cdot)\bigr)$
is a subbrace of $B$.
Since $(P,+)\trianglelefteq(B,+)$, its additive
normalizer is $(B,+)$.
Hence this intersection is precisely
$\N_{(B,\cdot)}\bigl((P,\cdot)\bigr)$.

\medskip
\noindent
\textup{(2)}
Write $\left|\N_{(B,\cdot)}\bigl((P,\cdot)\bigr)\right|
 =q^a m,$ where $q\nmid m.$
By \textup{(1)}, the additive quotient map restricts
to a homomorphism
\[
 \begin{aligned}
 \pi:
 \left(\N_{(B,\cdot)}\bigl((P,\cdot)\bigr),+\right)
 &\longrightarrow (B,+)/(Q,+),\\
 b&\longmapsto b+Q.
 \end{aligned}
\]
Its kernel is $ \ker\pi=
 \N_{(B,\cdot)}\bigl((P,\cdot)\bigr)\cap Q.$
Since the normalizer and $Q$ are subbraces, this intersection is also a multiplicative subgroup.
This is a $q$-group, whereas $\operatorname{im}\pi$
has order prime to $q$.
The first isomorphism theorem therefore gives
$|\ker\pi|=q^a$ proving \textup{(2)}.

\medskip
\noindent
\textup{(3)}
Take $ t\in
 \N_{(B,\cdot)}\bigl((P,\cdot)\bigr)\cap(Q,\cdot)$ and $x \in P$
and put $y=t\cdot x\cdot t^{-1}$.
Then $y\cdot t=t\cdot x$ and so $ y+\lambda_y(t)=t+\lambda_t(x).$
Note that $ y,\lambda_t(x)\in P$ and $ t,\lambda_y(t)\in Q.$
Since $(P,+)$ and $(Q,+)$ centralize each other,
we may rewrite the equality as $ y+\lambda_y(t)=\lambda_t(x)+t.$
Comparing the unique components in $P$ and $Q$
gives $y=\lambda_t(x)$, proving
\eqref{eq:normalizer-action}.
Thus $\lambda_t|_P$ is an additive automorphism
and agrees with a multiplicative automorphism
of $(P,\cdot)$.
It is therefore an automorphism of the skew brace $P$.
\end{proof}
\subsection{Constructing the quotient character}\label{subsec:quotient-character}
We first record the group-theoretic lemma used to make the character descend to the multiplicative quotient.
\begin{lemma}\label{lem:Frattini-character}
Let $G$ be a finite group, let $q$ be a prime, and let $A$ be an abelian $q$-group. Let $\psi:G\to A$ be a homomorphism and $R\trianglelefteq G$. Suppose $p\ne q$, $P\in\Syl_p(G)$ and $Q\in\Syl_q(\N_G(P))$. If $\psi$ is trivial on $R\cap Q$, then it is trivial on $R$.
\end{lemma}
\begin{proof}
Since $\psi(P)=1$, the subgroup $P$ is a Sylow $p$-subgroup of $RP\cap\ker\psi$. The latter is normal in $RP$, so the Frattini argument (see \cite[Theorem~5.2.14]{Robinson}) yields $RP=(RP\cap\ker\psi)\N_{RP}(P)$. Since $\N_{RP}(P)=\N_R(P)P$, we obtain
\[
\psi(R)=\psi(RP)=\psi(\N_{RP}(P))=\psi(\N_R(P)).
\]
If this image were nontrivial, choose $u\in\N_R(P)$ with $\psi(u)\ne1$. Write $|u|=q^a m$, where $(q,m)=1$, and set $v=u^m$. Then $v$ is a $q$-element of $\N_R(P)$ and $\psi(v)=\psi(u)^m\ne1$. By Sylow's theorem, some conjugate $t$ of $v$ by an element of $\N_G(P)$ belongs to $Q$. The normality of $R$ gives $t\in R\cap Q$, whereas $\psi(t)=\psi(v)\ne1$, a contradiction.
\end{proof}
\begin{definition}\label{def:determinant-character}
Let $B$ be a finite skew brace, let $p$ be a prime
dividing $|B|$, and suppose that $(P,+)$ is a normal
Sylow $p$-subgroup of $(B,+)$.
Let $R$ be a soluble normal subgroup of $(B,\cdot)$.
Define
\[
 \chi_{R,p}:(B,\cdot)\longrightarrow\F_p^\times,
 \qquad
 \chi_{R,p}(b)=
 \Delta_{(P,+)}(\lambda_b|_P)\,
 \beta_{R,p}(b)^{-1}.
\]
\end{definition}

Since $(P,+)$ is characteristic in $(B,+)$,
the restriction $\lambda_b|_P$ belongs to
$\Aut(P,+)$ for every $b\in B$.
Thus $\chi_{R,p}$ is well defined.
Both factors in its definition are homomorphisms,
so $\chi_{R,p}$ is a homomorphism.

\begin{lemma}\label{lem:quotient-character}
Let $B$ be a finite skew brace, and let $p,q$ be distinct prime divisors of $|B|$ such that $q\mid p-1$. Suppose that $(P,+)$ and $(Q,+)$ are normal Sylow $p$- and $q$-subgroups of $(B,+)$, respectively. Let $R$ be a soluble normal subgroup of $(B,\cdot)$. Then, for every $t\in\N_{(B,\cdot)}((P,\cdot))\cap(Q,\cdot)$,
\begin{equation}\label{eq:general-local-character}
\chi_{R,p}(t)=\Delta_{(P,\cdot)R/R}(\Ad_{tR}).
\end{equation}
Write $p-1=q^a m$, where $a\geq1$ and $\gcd(q,m)=1$. The character $\psi_{R,p,q}=\chi_{R,p}^{\,m}$ is trivial on $R$ and induces
\[
\overline\psi_{R,p,q}:(B,\cdot)/R\longrightarrow\F_p^\times,
\qquad
\overline\psi_{R,p,q}(bR)=\chi_{R,p}(b)^m.
\]
\end{lemma}
\begin{proof}
Take $t\in\N_{(B,\cdot)}((P,\cdot))\cap(Q,\cdot)$. By Lemma~\ref{lem:brace-normalizer}, the restriction $\lambda_t|_P$ is an automorphism of the skew brace $P$ and agrees with multiplicative conjugation by $t$. Proposition~\ref{prop:p-brace-determinant} and Lemma~\ref{lem:soluble-normal-character} give
\[
\begin{aligned}
\Delta_{(P,+)}(\lambda_t|_P)
&=\Delta_{(P,\cdot)}(\Ad_t|_P)\\
&=\beta_{R,p}(t)\,\Delta_{(P,\cdot)R/R}(\Ad_{tR}).
\end{aligned}
\]
Dividing by $\beta_{R,p}(t)$ proves \eqref{eq:general-local-character}.
Since $\psi_{R,p,q}(b)^{q^a}=\chi_{R,p}(b)^{p-1}=1$ for every $b\in B$, the image of $\psi_{R,p,q}$ is an abelian $q$-group. If, in addition, $t\in R$, then $tR=R$, so \eqref{eq:general-local-character} gives $\chi_{R,p}(t)=1$, and hence $\psi_{R,p,q}(t)=1$. Thus $\psi_{R,p,q}$ is trivial on $R\cap\N_{(B,\cdot)}((P,\cdot))\cap(Q,\cdot)$. By Lemma~\ref{lem:brace-normalizer}\textup{(2)}, the group $\N_{(B,\cdot)}((P,\cdot))\cap(Q,\cdot)$ is a Sylow $q$-subgroup of $\N_{(B,\cdot)}((P,\cdot))$. Lemma~\ref{lem:Frattini-character} therefore gives $R\leq\ker\psi_{R,p,q}$, yielding the stated quotient character.
\end{proof}
Summarizing the preceding results, we obtain the following construction of determinant characters and their factorization through the multiplicative quotient.
\begin{center}
\setlength{\fboxsep}{9pt}
\fbox{\begin{tikzpicture}[every node/.style={align=center}]
\node (A) at (0,0)
{$R\trianglelefteq(B,\cdot)$ soluble,
$\ (P,+)\trianglelefteq(B,+)$};
\node (B) at (0,-1.15)
{$\chi_{R,p}:(B,\cdot)\longrightarrow\F_p^\times$};
\node (C) at (0,-2.9)
{$\psi_{R,p,q}=\chi_{R,p}^{\,m}:
(B,\cdot)\longrightarrow\F_p^\times$};
\node (D) at (0,-4.15)
{$\overline\psi_{R,p,q}:
(B,\cdot)/R\longrightarrow\F_p^\times$};
\draw[->,thick] (A.south) -- (B.north);
\draw[->,thick] (B.south) --
node[right,font=\scriptsize]
{$q\mid p-1,\quad (Q,+)\trianglelefteq(B,+)$\\
$[P,Q]_+=\{0\}$}
(C.north);
\draw[->,thick] (C.south) --
node[right,font=\scriptsize]
{$R\leq\ker\psi_{R,p,q}$}
(D.north);
\end{tikzpicture}}
\end{center}

\medskip 

We first recall a group-theoretic fact about normalizers in quotients, which will be used to evaluate the quotient character on the normalizer of a multiplicative Sylow subgroup.
\begin{lemma}\label{lem:normalizer-quotient}
Let $R\trianglelefteq G$ and $P\in\Syl_p(G)$. The quotient map induces a surjective homomorphism
\[
\varphi:\N_G(P)\longrightarrow\N_{G/R}(PR/R),\qquad g\longmapsto gR.
\]
\end{lemma}
\begin{proof}
Let $gR\in\N_{G/R}(PR/R)$. Then $g(PR)g^{-1}=PR$, so $gPg^{-1}$ and $P$ are Sylow $p$-subgroups of $PR$. Hence some element $ur\in PR$, with $u\in P$ and $r\in R$, conjugates $gPg^{-1}$ to $P$. Since $u$ normalizes $P$, also $r$ conjugates $gPg^{-1}$ to $P$. Consequently $rg\in\N_G(P)$ and $(rg)R=gR$, proving surjectivity.
\end{proof}
\begin{lemma}\label{lem:quotient-normalizer-character}
Assume the hypotheses and notation of Lemma~\ref{lem:quotient-character}. For every $z\in\N_{(B,\cdot)/R}((P,\cdot)R/R)$, we have
\begin{equation}\label{eq:normalizer-quotient-character}
\overline\psi_{R,p,q}(z)=\Delta_{(P,\cdot)R/R}(\Ad_z)^m.
\end{equation}
\end{lemma}
\begin{proof}
Recall that $p-1=q^a m$, where $\gcd(q,m)=1$. Define
\[
\begin{aligned}
\Phi:\N_{(B,\cdot)/R}((P,\cdot)R/R)
&\longrightarrow\{u\in\F_p^\times:u^{q^a}=1\},\\
z&\longmapsto
\overline\psi_{R,p,q}(z)\,
\Delta_{(P,\cdot)R/R}(\Ad_z)^{-m}.
\end{aligned}
\]
Both factors are homomorphisms into the indicated abelian $q$-group, so $\Phi$ is a homomorphism. By Lemma~\ref{lem:normalizer-quotient}, the quotient map induces a surjection
\[
\N_{(B,\cdot)}((P,\cdot))
\longrightarrow
\N_{(B,\cdot)/R}((P,\cdot)R/R),
\qquad t\longmapsto tR.
\]
By Lemma~\ref{lem:brace-normalizer}\textup{(2)}, the subgroup $\N_{(B,\cdot)}((P,\cdot))\cap(Q,\cdot)$ is a Sylow $q$-subgroup of $\N_{(B,\cdot)}((P,\cdot))$. Its image is therefore a Sylow $q$-subgroup of the domain of $\Phi$. For every $t\in\N_{(B,\cdot)}((P,\cdot))\cap(Q,\cdot)$, Lemma~\ref{lem:quotient-character} gives
\[
\Phi(tR)=
\chi_{R,p}(t)^m\,
\Delta_{(P,\cdot)R/R}(\Ad_{tR})^{-m}=1.
\]
Thus $\ker\Phi$ contains a Sylow $q$-subgroup of its domain, so its index is prime to $q$. This index equals $|\operatorname{im}\Phi|$, which is a power of $q$. Hence $\Phi$ is trivial, proving \eqref{eq:normalizer-quotient-character}.
\end{proof}

\section{Proof of the main obstruction}\label{sec:obstruction}

The following lemma provides the element and the determinant
calculation needed to complete the proof of
Proposition~\ref{prop:obstruction}.

\begin{lemma}\label{lem:wreath-determinant}
Let $r\geq1$, $T\leq S_r$, and let $W=\GL_3(2)^r\rtimes T$, with $T$ permuting the factors. For every $P\in\Syl_7(W)$, there exists an element $z\in\GL_3(2)^r\cap\N_W(P)$ of order $3$ such that $\Delta_P(\Ad_z)=2$ in $\F_7^\times$.
\end{lemma}
\begin{proof}
We first prove the lemma for a particular Sylow $7$-subgroup. Regard $\F_8$ as a $3$-dimensional vector space over $\F_2$. Choose a generator $\xi$ of $\F_8^\times$ and define $u:\F_8\longrightarrow\F_8$ by $u(x)=\xi x$ for every $x\in\F_8$.
Then $u\in\GL_3(2)$ and $|u|=7$. Since $|\GL_3(2)|=168$, the subgroup $U=\langle u\rangle$ is a Sylow $7$-subgroup of $\GL_3(2)$. Choose $V\in\Syl_7(T)$ and put $P_0=U^r\rtimes V$. This subgroup is well defined because $T$ permutes the coordinates of $U^r$. Moreover,
\[
|P_0|=7^r|V|=7^r|T|_7=|W|_7,
\]
so $P_0\in\Syl_7(W)$.
 Consider the $\F_2$-linear Frobenius map $s(x)=x^2$. Since $s^3(x)=x^8=x$ and $s\neq1$, it has order $3$, with $s^{-1}(x)=x^4$. For every $j\in\mathbb Z$
and $x\in\F_8$ we have
\[
\begin{aligned}
(su^js^{-1})(x)
&=s(\xi^j x^4)=\xi^{2j}x^8=\xi^{2j}x=u^{2j}(x).
\end{aligned}
\]
Thus $s$ normalizes $U=\langle u\rangle$ and
$\Ad_s(u^j)=u^{2j}$. Since $U$ is cyclic of order $7$, we have $\Phi(U)=1$.
Hence the only nontrivial Frattini factor is $U/1$.
Consider the isomorphism
\[
\theta:U\longrightarrow(\F_7,+),
\qquad \theta(u^j)=[j]_7,
\]
and the automorphism
\[
\mu_2:(\F_7,+)\longrightarrow(\F_7,+),
\qquad \mu_2(a)=2a.
\]
For every $u^j\in U$, we have
\[
\theta(\Ad_s(u^j))
=\theta(u^{2j})
=[2j]_7
=\mu_2([j]_7)
=\mu_2(\theta(u^j)).
\]
Hence $\theta\Ad_s=\mu_2\theta$, or equivalently, $\theta\Ad_s\theta^{-1}=\mu_2.$
Lemma~\ref{lem:Delta-series}\textup{(1)}, applied to
the isomorphism $\theta$, therefore gives
\[
\Delta_U(\Ad_s)
=\Delta_{(\F_7,+)}(\theta\Ad_s\theta^{-1})
=\Delta_{(\F_7,+)}(\mu_2).
\]
Since $\Phi((\F_7,+))=\{0\}$, the definition of $\Delta$
now yields
\[
\Delta_{(\F_7,+)}(\mu_2)
=\det_{\F_7}(\mu_2)
=\det_{\F_7}(2)
=2.
\]
Thus $\Delta_U(\Ad_s)=2$. Consider the natural permutation action of
$V\le T\le S_r$ on $\{1,\ldots,r\}$, given by
$v\cdot i=v(i)$.
Choose an orbit $\mathcal O$ of this action.
By the orbit--stabilizer theorem $|\mathcal O|=[V:\operatorname{Stab}_V(i)]=7^k$
for any $i\in\mathcal O$ and some $k\ge0$,
since $V$ is a $7$-group. Define
\[
z_0=(z_1,\ldots,z_r),
\qquad
z_i=
\begin{cases}
s,&i\in\mathcal O,\\
1,&i\notin\mathcal O.
\end{cases}
\]
Then $z_0$ has order $3$ and normalizes $U^r$. For every $v\in V$ we have
\[
vz_0v^{-1}
=(z_{v^{-1}(1)},\ldots,z_{v^{-1}(r)})
=(z_1,\ldots,z_r)=z_0,
\]
because $z_i$ is constant on each $V$-orbit. Thus $z_0$ centralizes $V$, and consequently $z_0\in\GL_3(2)^r\cap\N_W(P_0).$
For every $(u^{a_1},\ldots,u^{a_r})\in U^r$,
conjugation is computed coordinatewise:
\[
\begin{aligned}
\Ad_{z_0}(u^{a_1},\ldots,u^{a_r})
&=(z_1u^{a_1}z_1^{-1},\ldots,
   z_ru^{a_r}z_r^{-1})\\
&=(u^{d_1a_1},\ldots,u^{d_ra_r}),
\end{aligned}
\]
where
\[
d_i=
\begin{cases}
2,&i\in\mathcal O,\\
1,&i\notin\mathcal O.
\end{cases}
\]
Here we have used $su^as^{-1}=u^{2a}$.
Consider the group isomorphism
\[
\theta_r:U^r\longrightarrow(\F_7^r,+),
\qquad
(u^{a_1},\ldots,u^{a_r})
\longmapsto([a_1]_7,\ldots,[a_r]_7),
\]
and the linear automorphism
\[
\mu:\F_7^r\longrightarrow\F_7^r,
\qquad
\mu(a_1,\ldots,a_r)=(d_1a_1,\ldots,d_ra_r).
\]
For every $(u^{a_1},\ldots,u^{a_r})\in U^r$, we have
\[
\begin{aligned}
\theta_r\bigl(\Ad_{z_0}(u^{a_1},\ldots,u^{a_r})\bigr)
&=\theta_r(u^{d_1a_1},\ldots,u^{d_ra_r})\\
&=([d_1a_1]_7,\ldots,[d_ra_r]_7)\\
&=\mu([a_1]_7,\ldots,[a_r]_7)\\
&=\mu\bigl(\theta_r(u^{a_1},\ldots,u^{a_r})\bigr).
\end{aligned}
\]
Thus $\theta_r(\Ad_{z_0}|_{U^r})\theta_r^{-1}=\mu.$
By Lemma~\ref{lem:Delta-series}\textup{(1)},
\[
\begin{aligned}
\Delta_{U^r}(\Ad_{z_0}|_{U^r})
&=\Delta_{(\F_7^r,+)}
  \bigl(\theta_r(\Ad_{z_0}|_{U^r})\theta_r^{-1}\bigr)=\Delta_{(\F_7^r,+)}(\mu).
\end{aligned}
\]
Since $(\F_7^r,+)$ is elementary abelian, its Frattini
subgroup is $\{0\}$, so
\[
\Delta_{(\F_7^r,+)}(\mu)=\det_{\F_7}(\mu).
\]
The matrix of $\mu$ in the standard basis is
\[
\begin{pmatrix}
d_1& &0\\
&\ddots&\\
0& &d_r
\end{pmatrix}.
\]
Consequently,
\[
\begin{aligned}
\Delta_{U^r}(\Ad_{z_0}|_{U^r})
&=\det_{\F_7}\bigl(\operatorname{diag}(d_1,\ldots,d_r)\bigr)=\prod_{i=1}^r d_i=\left(\prod_{i\in\mathcal O}2\right)
  \left(\prod_{i\notin\mathcal O}1\right)=2^{|\mathcal O|}.
\end{aligned}
\]

We next compute the automorphism induced by $\Ad_{z_0}$
on $P_0/U^r$.
Every element of $P_0=U^r\rtimes V$ has the form $xv$,
with $x\in U^r$ and $v\in V$.
Since $z_0$ centralizes $V$ and normalizes $U^r$,
\[
z_0(xv)z_0^{-1}
=(z_0xz_0^{-1})(z_0vz_0^{-1})
=(z_0xz_0^{-1})v,
\qquad z_0xz_0^{-1}\in U^r.
\]
Therefore
\[
\overline{\Ad}_{z_0}(xvU^r)
=z_0(xv)z_0^{-1}U^r
=vU^r
=xvU^r.
\]
Hence
\[
\overline{\Ad}_{z_0}=\mathrm{id}_{P_0/U^r},
\qquad
\Delta_{P_0/U^r}(\overline{\Ad}_{z_0})=1.
\]
Applying Lemma~\ref{lem:Delta-series}\textup{(2)} to
the $\Ad_{z_0}$-invariant normal subgroup
$U^r\trianglelefteq P_0$, we obtain
\[
\begin{aligned}
\Delta_{P_0}(\Ad_{z_0})
&=\Delta_{U^r}(\Ad_{z_0}|_{U^r})\,
  \Delta_{P_0/U^r}(\overline{\Ad}_{z_0})=2^{|\mathcal O|}\cdot1=2^{7^k}.
\end{aligned}
\]
Finally, $7^k\equiv1\pmod3$ and $2^3=1$ in
$\F_7^\times$, so $2^{7^k}
=2(2^3)^{(7^k-1)/3}
=2.$
Thus $\Delta_{P_0}(\Ad_{z_0})=2$, as required. Finally, let $P\in\Syl_7(W)$ be the prescribed subgroup. Choose $w\in W$ such that $P=wP_0w^{-1}$ and put $z=wz_0w^{-1}$. Since $\GL_3(2)^r\trianglelefteq W$, we have $z\in\GL_3(2)^r$ and $|z|=3$. Also,
\[
zPz^{-1}
=wz_0P_0z_0^{-1}w^{-1}
=wP_0w^{-1}=P.
\]
By Lemma~\ref{lem:Delta-series}\textup{(1)},
\[
\Delta_P(\Ad_z)=\Delta_{P_0}(\Ad_{z_0})=2,
\]
as required.
\end{proof}

We are now ready to prove Proposition~\ref{prop:obstruction}.

\begin{proof}[Proof of Proposition~\ref{prop:obstruction}]
Suppose that such a soluble normal subgroup $R$ exists, and identify $(B,\cdot)/R$ with $\GL_3(2)^r\rtimes T$. Let $(P,+)$ and $(Q,+)$ be the normal additive Sylow $7$- and $3$-subgroups, respectively. Then $(P,\cdot)R/R$ is a Sylow $7$-subgroup of $(B,\cdot)/R$.
Apply Lemma~\ref{lem:quotient-character} with $(p,q)=(7,3)$. Since $7-1=3\cdot2$, the character $\psi_{R,7,3}=\chi_{R,7}^{\,2}$ is trivial on $R$ and induces $\overline\psi_{R,7,3}:(B,\cdot)/R\to\{1,2,4\}$. By Lemma~\ref{lem:wreath-determinant}, there is an element $z\in\GL_3(2)^r\cap\N_{(B,\cdot)/R}((P,\cdot)R/R)$
with $\Delta_{(P,\cdot)R/R}(\Ad_z)=2$. Since $\GL_3(2)^r$ is perfect, $z$ lies in the derived subgroup of $(B,\cdot)/R$ and $\overline\psi_{R,7,3}(z)=1$. By Lemma~\ref{lem:quotient-normalizer-character}, however,
\[
1=\overline\psi_{R,7,3}(z)=\Delta_{(P,\cdot)R/R}(\Ad_z)^2=2^2=4
\]
in $\F_7^\times$, a contradiction.
\end{proof}

\section*{Declaration on the use of artificial intelligence}

Theorem \ref{thm:A} originated from the author's attempt to
prove that every finite skew brace with supersoluble
additive group has soluble multiplicative group.
The proof emerged from a continuous and extensive
interaction with ChatGPT (Sol 6.1), lasting several
months, during which the author guided the mathematical
investigation through questions, discussions, and
successive refinements of the arguments.
A crucial turning point was the idea of using the
iterated Frattini series of finite $p$-groups to
construct determinant characters, which ultimately
provided the key tool for establishing the main result. The initial proof generated by ChatGPT was incomprehensible in its original form and could not be presented as a readable mathematical argument without substantial reconstruction. The author subsequently undertook an extensive effort to understand, reconstruct, reorganize, and rewrite the arguments from the ground up, with the aim of making every step of the reasoning rigorous, transparent, and accessible to the mathematical community. All necessary mathematical verifications were carried out by the author, who takes full responsibility for the correctness of the results and the final version of the manuscript.

\end{document}